\documentclass{article}
\pdfoutput=1 
   \usepackage{authblk}
   \usepackage{amsthm}
   \usepackage{amsmath}
   \usepackage{amsfonts}    
   \usepackage{amssymb}    

	\newtheorem{theorem}{Theorem}[section]
	
	\newtheorem{lemma}{Lemma}[section]
	\newtheorem{claim}{Claim}[section]
	\newtheorem{algorithm}{Algorithm}[section]

	\theoremstyle{definition}
	\newtheorem{definition}{Definition}[section]

	\theoremstyle{remark}

	\numberwithin{equation}{section}

\begin{document}


%
%
%
%
\def\put#1 at (#2,#3){\vbox to 0pt{\kern#3
                     \dimen2=#2
                     \hbox{\kern\dimen2{#1}}\vss}\nointerlineskip}

\def\dim{\hbox{Dim}}
\def\rank{\hbox{rank}}
\def\ret{\hfill \cr}
\def\qu#1{\hskip #1em\relax}
\def\Cal{\mathcal}
%
%
\def\bbc{\Bbb C}
%
%
\def\bbr{\Bbb R}
%
%
\def\intersect{\mathop{\cap}\displaylimits}
%
%
\def\sinc{\hbox{sinc}}
%
%
\def\twoprime{{\prime\prime}}
%
%
\def\threeprime{{\prime\prime\prime}}
%
%
\def\union{\mathop{\cup}\displaylimits}
%
%
\def\tten#1{$\times{10}^{#1}\qu{1}$}
%
%

\def\tcap#1#2{
\vskip -25pt
{\par\centering \small {\bf Table #1} #2\par}
\vskip 20pt
}
\def\tcapns#1#2{
{\par\centering \small {\bf Table #1} #2\par}
}


\def\smartqed{}

\newcommand{\sbt[1]}{\,\begin{picture}(-1,1)(-1,-#1)\circle*{#1}\end{picture}\ }

\newcommand{\abs}[1]{\lvert#1\rvert}

\newcommand{\blankbox}[2]{%
  \parbox{\columnwidth}{\centering
    \setlength{\fboxsep}{0pt}%
    \fbox{\raisebox{0pt}[#2]{\hspace{#1}}}%
  }%
}

\smartqed

\title{A Class of Algorithms for Quadratic Minimization}

\author{Marc Stromberg\break{\sl email: mstromberg@psmfc.org}}

\maketitle

\abstract{Certain problems in quadratic minimization can be reduced to finding the point $x$ of a polyhedron ${\Cal P}$ that minimizes the distance $\|x-p\|$ for some $p\notin {\Cal P}$. This amounts to a search for the appropriate face $F$ of ${\Cal P}$ for which the minimizing point is the projection of $p$ onto $F$. We present a class of algorithms for finding the face $F$ and the corresponding minimizing point $x\in {\Cal P}$, then a number of examples using those methods.}

\vskip 10 pt
\noindent{\bf 2020 Mathematics Subject Classification: }  65K99, 90C20 \hfill\break

\section{Introduction}

Let ${\Cal P}$ be a (convex) polyhedron. For faces $F$, $K$ of ${\Cal P}$ or ${\Cal P}$ itself we will write $F\leq K$ if $F$ is a face of $K$ and $F<K$ if $F$ is a proper face of $K$. If $X$ is any set we will write ${\Cal A}(X)$ for the affine hull of $X$ and ${\Cal C}(X)$ for the convex hull of $X$. We will assume for this discussion that ${\Cal P}$ has a representation ${\Cal P}=\{x\in\bbr^n\mid Ax \leq b\}$ where  the system in question is irredundant (see \cite{as-tlip} for a full discussion of this type of representation). We will also assume that the system $Ax \leq b$ has no implicit equalities, so that ${\Cal P}$ is full-dimensional in the ambient space $\bbr^n$ (or equivalently, restricting things to ${\Cal A}({\Cal P})$). 
This is not a real restriction because the entire discussion can be situated in the affine subspace determined by implicit equalities, in which ${\Cal P}$ is full-dimensional, if there are such equalities, and for $p\notin {\Cal P}$ we can restrict to the projection of $p$ onto this affine subspace. In any case, {\lq the full dimensional case\rq}  will be assumed to mean that $p\in{\Cal A}({\Cal P})$, but it is easy to adapt Algorithm \ref{al:qm1p1} below to include the nonfull-dimensional case.
The main results will be stated in terms of the following definition.

\begin{definition}    
\label{de:qm1p1}
Let $e\in F < {\Cal P}$ and let $p \notin{\Cal P}$. An {\it escape} from $e$ toward $p$ along $F$ is a point $e^\prime \neq e$ such that \begin{equation}e^\prime = e+t_0(\pi_F-e)\label{eq:qm1p0}\end{equation} where $\pi_F$ is the orthogonal projection of $p$ onto ${\Cal A}(F)$ and where $t_0\in (0, 1]$ is largest such that $e+t_0(\pi_F-e)\in {\Cal P}$.
\end{definition}

If there are no escapes from $e$ then there is no path in ${\Cal P}$, at least locally, from $e$ to a point nearer to $p$. This intuitively suggests that $e$ minimizes the distance to $p$. The proof of Theorem \ref{th:qm1p1} below provides a basis for the intuition. 

Given a point $e\in {\Cal P}$ and $p\notin {\Cal P}$ we will say that $e$ is {\it $p$-visible} if there are no points of ${\Cal P}$ on the open segment $(e, p)$. Given our full-dimension restriction, points of ${\Cal P}$ that are $p$-visible for some $p\notin {\Cal P}$ will necessarily belong to some facet of ${\Cal P}$. 

The proof of Theorem \ref{th:qm1p1} will depend on the following constructions. We assume that $p\notin{\Cal P}$ is fixed but otherwise arbitrary. Let $\nu\in {\Cal P}$ be the point of ${\Cal P}$ nearest to $p$. If $x \neq \nu$ and $x\in {\Cal P}$ is $p$-visible and if $t\in [0, 1]$ let $\ell_{x, \nu}(t) = (1-t)x+t\nu$, and define $\rho_{x, \nu}(t)$ as the point $(1-s)\ell_{x, \nu}(t) + sp$ where $s\in[0,1]$ is largest such that $(1-s)\ell_{x, \nu}(t) + sp\in {\Cal P}$. The set of points $\rho_{x, \nu}(t)$ for $t\in [0,1]$ will be called the {\it $p$-visible path} from $x$ to $\nu$. Finally let $\delta_{x, \nu}(t) = \|\rho_{x, \nu}(t)-p\|$ for each $t$.  This definition does not require ${\Cal P}$ to be full-dimensional, i.e.,  it applies if $p$ does not lie in ${\Cal A}({\Cal P})$. If $\dim({\Cal P}) = 1$ then the $p$-visible path for $p\in {\Cal A}({\Cal P})$ degenerates to a point since it is not possible to pick $x\neq \nu$ such that both are $p$-visible.
\begin{lemma}
\label{le:qm1p1}
The function $\delta_{x, \nu}(t)$ is  continuous and strictly decreasing on $[0,1]$.\end{lemma}\begin{proof}  If ${\Cal P}$ is one-dimensional then a $p$-visible path is only nondegenerate if $p\notin {\Cal A}({\Cal P})$, in which case it is just the segment from $x$ to $\nu$. In general we may consider the polyhedron consisting of ${\Cal P}_{x, \nu} = {\Cal C}(\{x, \nu, p\}) \intersect {\Cal P}$, in other words, the intersection of ${\Cal P}$ with the triangle (i.e., polytope) determined by $x, \nu, p$ (these points are not colinear except in the one-dimensional case). This is an at most two-dimensional polytope whose facets are line segments, or it reduces to the line segment $[x, \nu]$. In the latter case, $\nu$ is the nearest point to $p$ along the segment, so by a law of cosines argument $\delta_{x, \nu}(t)$ is strictly decreasing from $x$ to $\nu$. Otherwise, ${\Cal P}_{x, \nu}$ is two-dimensional and it is clear that the $p$-visible path from $x$ to $\nu$ consists of the piecewise linear path formed from facets and vertices of ${\Cal P}_{x, \nu}$ connecting $x$ and $\nu$ along its boundary. Specifically, the $p$-visible path from $x$ to $\nu$ consists of line segments between vertices $v_0=x, v_1, \ldots, v_m=\nu$ of ${\Cal P}_{x, \nu}$ for some $m\geq 1$. Consider the triangle $\{v_{m-1}, v_m, p\}$. For $y\in [v_{m_1}, v_m]$ as $y$ moves toward $v_m = \nu$, 
the length of the side $yp$ of triangle $\{y, \nu, p\}$ is strictly decreasing by applying the law of cosines. Explicitly if $y=y(t) = (1-t)v_{m-1}+tv_m$ and if $\theta_m$ designates the angle $v_{m-1} v_m p$ then setting $D_t = \|y(t)-p\|$  and $\delta = \|v_m-p\|$ we have
$$D_t^2 - \delta^2 =  (1-t)^2\|v_{m-1}-v_m\|^2 -2\delta(1-t)\|v_{m-1}-v_m\|\cos \theta_m\geq 0,$$
which shows that $$(1-t)\|v_{m-1}-v_m\|-2\|v_{m-1}-v_m\|\|v_m-p\|\cos \theta_m \geq 0$$
for all $t\in(0,1)$ thus $\cos \theta_m \leq 0$, so in fact $\pi/2\leq\theta_m< \pi$ which is consistent with intuition in any case, the second inequality in force because the points under consideration are $p$-visible.
Then
$$D_tD_t^\prime = \delta\|v_{m-1}-v_m\|\cos\theta_m - (1-t)\|v_m-v_{m-1}\|^2<0,$$ 
which shows $D_t$ is strictly decreasing for $t\in (0,1)$. Proceeding by  induction, the angle $\theta_{i-1}$ (angle $v_{i-2} v_{i-1} p$) is necessarily larger than the angle $\theta_i$ (angle $v_{i-1} v_i p$) for $i=2, \ldots, m$ for any of these that exist, or the convexity of ${\Cal P}_{x, \nu}$ would be violated. A similar argument now using the fact that $\theta_{i-1} > \theta_i$ for $i=1, \ldots, m$ shows that $\|y-p\|$ is strictly decreasing as a function of $t$ for $y=y(t)=(1-t)v_{i-1}+tv_i$ for $i=1, \ldots, m$, so ultimately $\delta_{x, \nu}(t)$ is strictly decreasing, continuous and even piecewise differentiable by construction.\end{proof}  

We note that the $p$-visible path from $x$ to $\nu$ consists of the facets of ${\Cal P}_{x, \nu}$ that are contained in facets of ${\Cal P}$, in the full-dimensional case, and is just the segment $[x, \nu]$ in the nonfull-dimensional case, as follows.

If $x, \nu\in {\Cal P}$ and $p\notin {\Cal A}({\Cal P})$, a point $y$ on the $p$-visible path has the form $y(s, t)=(1-s)\ell_{x, \nu}(t) + sp$ where $t\in [0,1]$ and $s$ is largest so that this point is in ${\Cal P}$. But for fixed $t$ we have $y_0=y(0,t)\in {\Cal P}$, and if it were true that $y_s=y(s, t)\in {\Cal P}$ for some $s>0$ then the entire line containing $y_0$ and $y_s$ would lie in ${\Cal A}({\Cal P})$, thus so would $p$, contradicting that $p\notin {\Cal A}({\Cal P})$. Thus in this case the $p$-visible path from $x$ to $\nu$ consists of just the segment $[x, \nu]$.
On the other hand, it is clear that in the full-dimensional case (meaning that $p\in {\Cal A}({\Cal P})$)  distinct segments of a $p$-visible path will lie in separate facets of ${\Cal P}$.  It is easy enough to express constraints for ${\Cal P}_{x, \nu}$, consisting of constraints $C$ that define ${\Cal C}(\{x, \nu, p\}) $ together with the constraints that define ${\Cal P}$. The constraints $C$ are satisfied by all of the  points on the $p$-visible path from $x$ to $\nu$, so at a transition between segments of this path, i.e., facets of ${\Cal P}_{x, \nu}$, the only constraint that can possibly change is one of the constraints defining ${\Cal P}$, so the transition is also between facets of ${\Cal P}$. 

Let polyhedron ${\Cal P} = \{x\in \bbr^m\mid {\hat A}x\leq {\hat b}, {\tilde A}x = {\tilde b}\}$ where $\{x\in \bbr^m\mid {\tilde A}x = {\tilde b}\}$ is an $n$-dimensional affine subspace of $\bbr^m$ and where ${\hat A}x\leq {\hat b}$ has no implicit inequalities. Given a particular solution $\xi_0$ of ${\tilde A}x = {\tilde b}$ we have ${\Cal P}-\xi_0 = \{x\in \bbr^m\mid {\hat A}x\leq { b}, {\tilde A}x = 0\}$ where $b={\hat b}-{\hat A}\xi_0$. We must have $n=m-r$ where $r$ is the rank of ${\tilde A}$ and $n$ is the dimension of the null space $N({\tilde A})$ of ${\tilde A}$. Choose an orthonormal basis $\nu_0, \dots, \nu_{n-1}$ for $N({\tilde A})$ and complete this to an orthonormal basis of $\bbr^m$ by the addition of vectors $\gamma_0, \dots, \gamma_{m-n-1}$. The matrix $\Lambda$ with columns $\nu_0, \dots, \nu_{n-1},\gamma_0, \dots, \gamma_{m-n-1}$ then defines an invertible  linear transformation $\Lambda:\bbr^m\rightarrow \bbr^m$ by $x\rightarrow \Lambda x$. Define the injection $\iota_{n, m}:\bbr^n\rightarrow \bbr^m$ as $\iota_{n, m}(x)=(x_0, \dots,x_{n-1}, 0, \dots, 0)$ and the projection $\pi_{m,n}:\bbr^m\rightarrow\bbr^n$ as $\pi_{m, n}(x) = (x_0, \dots, x_{n-1})$. Denoting translation by $a$ as $\tau_a$, it is clear that the map $I=\tau_{\xi_0}\circ\Lambda\circ\iota_{n, m}:\bbr^n\rightarrow {\Cal A}({\Cal P})$ is invertible with inverse the restriction of $P=\pi_{m,n}\circ\Lambda^*\circ\tau_{-\xi_0}$ to ${\Cal A}({\Cal P})$, and we have $P({\Cal P}) = {\Cal Q} = \{x\in \bbr^n\mid Ax \leq b\}$ where we set  ${\bar A} ={\hat A}\Lambda$ and finally define matrix $A$ as comprised of the first $n$ columns of ${\bar A}$. Moreover given matrix $A$ we can extend this matrix to a matrix ${\bar A}^\prime$ by adding $m-n$ columns in an arbitrary way, and then set ${\hat A}^\prime = {\bar A}^\prime \Lambda^*$. If $x\in N({\tilde A})$ we can write $x=\Lambda y$ for some $y\in \iota_{n,m}(\bbr^n)$. Then ${\hat A}^\prime x = {\bar A}^\prime \Lambda^*\Lambda y = A\pi_{m,n}(y)$ so
for $x\in N({\tilde A})$ we have ${\hat A}x = {\hat A}^\prime x$ regardless of how we extend $A$, and can construct ${\Cal P}$ from ${\Cal Q}$. If ${\hat A}^\prime$ and ${\hat A}$ are both obtained by extending $A$ and $\xi_0$ is a particular solution as above then ${\hat A}^\prime x \leq {\hat b}^\prime$ iff ${\hat A}x\leq {\hat b}$ for all $x$ satisifying ${\tilde A}x = {\tilde b}$ where ${\hat b}^\prime = b+{\hat A}^\prime\xi_0$ and ${\hat b} = b+{\hat A}\xi_0$. Therefore we can always obtain ${\Cal P}$ from ${\Cal Q}$ and the system ${\tilde A}x= {\tilde b}$, i.e., even though the inequality constraints for ${\Cal P}$ may differ, we obtain the same ${\Cal P}$. In any case with a fixed choice of $\xi_0$ the maps $I$ and $P$ are inverses of each other, and both preserve distance and angles, since $\Lambda$ is an orthogonal transformation. It is also true that $Ax\leq b$ is irredundant iff ${\hat A}x\leq {\hat b}$ is and that these maps preserve quantities like orthogonal projection from ${\Cal A}({\Cal P})$ to an affine subspace ${\Cal A}$, and escapes along affine subspaces. This process admits at least one adaptation of the eventual algorithm to the case in which there are implicit equalities.

The following theorem concerns the full-dimensional case $p\in {\Cal A}({\Cal P})$ or equivalently that ${\Cal P}$ has the same dimension as the affine hull of  $\{p\}\union {\Cal P}$. The argument for ${\Cal P} = \{x\in \bbr^m\mid { A}x\leq { b}, {\tilde A}x = {\tilde b}\}$ where $ {A}x\leq { b}$ has no implicit inequalities and ${\Cal A}({\Cal P}) = \{x\in \bbr^m\mid  {\tilde A}x = {\tilde b}\}$ is an $n$-dimensional affine subspace of $\bbr^m$ is the same as for the case that $m=n$ and the system ${\tilde A}x = {\tilde b}$  is empty and ${\Cal A}({\Cal P}) = \bbr^n$.
\begin{theorem}
\label{th:qm1p1} 
Let $p\in {\Cal A}({\Cal P})\setminus {\Cal P}$. A point $e\in {\Cal P}$  is the nearest point of ${\Cal P}$ to $p$ if and only if $e$ is $p$-visible and for all $F$ with $e\in F<P$ there is no escape from $e$ toward $p$ along $F$.
\end{theorem}

\begin{proof}

If $e$ is the nearest point to $p$ then since an escape would be closer to $p$, there can be no escapes, and certainly $e$ is  $p$-visible in this case.
So suppose  $e$ is $p$-visible and there are no escapes. Suppose also that $e$ is not the nearest point of ${\Cal P}$ to $p$. 

\begin{claim}
\label{cl:clmp1}
If $e$ is $p$-visible and is the nearest point of $F$ to $p$ for all ${\Cal P} > F\ni e$ then $e$ is the nearest point of ${\Cal P}$ to $p$.
\end{claim}
\begin{proof}
We may assume that ${\Cal P}$ is not a single point, in which case the claim is vacuously true. If ${\Cal P}$ is one-dimensional then $e$ is a vertex and is also the nearest point of ${\Cal P}$ to $p$ by our assumption that ${\Cal P}$ is full-dimensional and the fact that $e$ is $p$-visible, so we may assume that $\dim({\Cal P})\geq 2$. Let $\nu$ be the nearest point to $p$, assume $e\neq \nu$ and consider the $p$-visible path from $e$ to $\nu$. Since $e$ is $p$-visible, $e$ lies in some facet of ${\Cal P}$. The  $p$-visible path traverses facets of  ${\Cal P}$, 
so for some facet $F$ of ${\Cal P}$, a portion of the $p$-visible path (namely the first segment of the  $\pi_F$-visible path from $e$ to $\nu_F$) lies in $F\ni e$, and then for some $f\in F$ distinct from $e$ the segment $[f, e]$ lies in $F$ and in the $p$-visible path from $e$ to $\nu$. But for $F$ as a polyhedron the $p$-visible path in $F$ from $f$ to $e$ coincides with this segment (since $p$, $e$, $\nu$ and $f$ as well as the rest of the $p$-visible path in ${\Cal P}$ are coplanar). But then the distance to $p$ along this segment is both strictly increasing and decreasing, which contradiction shows that we must have $e = \nu$.
\end{proof}

By the claim, if we suppose that $e$ is not the nearest point to $p$, there is a face $F<{\Cal P}$ for which $e\in F$ is not the nearest point of $F$ to $p$. We also have $\dim(F)\geq 1$, then. Note that $e$ is also not the nearest point of $F$ to $\pi_F$, so $e\neq \pi_F$. If $e\in F^\circ$ and $B_\epsilon\subset F^\circ$ is a ball of radius $\epsilon$ centered at $e$, then for some $e^\prime\neq e$ the line segment $[e,  e^\prime]\subset B_\epsilon\intersect [e,  \pi_F]$ which is all that is necessary to show existence of an escape along $F$.

Now if every face that contains $e$ contains a point closer to $p$ than $e$ let $K$ be minimal with respect to containing $e$. Then $e\in K^\circ$ and $K$  is not a vertex (nor is $e$) and we have an escape as just observed.

At this point we may assume that for some $F\ni e$, $e$ is not the nearest point of $F$ to $p$, but that this is not true of every face that contains $e$. Therefore we may assume that $F$ is minimal with respect to containing both $e$ and a point nearer to $p$ than $e$, and furthermore that $e\notin F^\circ$ or we are done. We may also assume that $\pi_F\notin F$ since otherwise we immediately have an escape along $F$, because by assumption $e\neq \pi_F$. We may further assume that $\dim(F)\geq 2$ since in the one-dimensional case if $e$ is $p$-visible, $e$ is the nearest point of $F$ to $p$.  In any case $e\in K$ for some facet $K<F$ and therefore $e$ is the nearest point of $K$ to $p$, and in fact $e$ is the nearest point to $p$ (and to $\pi_F$) for any facet of $F$ that contains $e$.  The segment $[e, \pi_F]$ has no point of $F$ except $e$, or we again have an escape along $F$. On the other hand the segment $[e, \nu_F]$ is contained entirely within $F$ where $\nu_F$ is the nearest point of $F$ to $p$. 
The triangle (polytope) with vertices $e, \pi_F, \nu_F$ either intersects $F$ in a two-dimensional polytope, for which as usual the boundary must consist of finitely many segments, or it intersects $F$ in a single line segment. Either way, within ${\Cal A}(F)$, the $\pi_F$-visible path from $e$ to $\nu_F$ consists of finitely many segments that pass through facets of $F$ which are themselves at least one-dimensional. In particular the first segment originating at $e$ is contained in a facet of $F$ that contains $e$,
which is then a facet of $F$ containing $e$ for which $e$ is not the nearest point to $\pi_F$ (or $p$), which is a contradiction since $F$ was supposed to be minimal. \end{proof}

In the context of Theorem \ref{th:qm1p1} a point $e\in{\Cal P}$ that is $p$-visible can be viewed simply as not having an escape along ${\Cal P}$ itself, that is along the affine hull ${\Cal A}({\Cal P})$, and the projection of $p$ in this case is simply $p$ since $p\in {\Cal A}({\Cal P})$. 
We will use this point of view subsequently in the following form, which is the nonfull-dimensional case.
\begin{theorem}
\label{th:qm1p11} 
Let ${\Cal P}$ be an $n$-dimensional polyhedron in $\bbr^m$ and $p\in\bbr^m$. 
Then a point $e\in{\Cal P}$ is the nearest point of ${\Cal P}$ to $p$ if and only if $e$ has no escape toward $p$  along ${\Cal P}$ or a face of ${\Cal P}$ .
\end{theorem}

\begin{proof}
Suppose $e\in {\Cal P}$ has no escape toward $p$ along ${\Cal P}$ or a face of ${\Cal P}$. Let $p_0$ be the projection of $p$ onto the affine hull ${\Cal A}({\Cal P})$. If $p_0\in {\Cal P}$ then we must have $e=p_0$ and then $e$ is clearly the nearest point of ${\Cal P}$ to $p$. If $p_0\notin {\Cal P}$ then there can be no points of ${\Cal P}$ between $e$ and $p_0$ on the line between them, so $e$ is $p_0$-visible and we are in the situation of Theorem  \ref{th:qm1p1}, restricting our viewpoint to ${\Cal A}({\Cal P})$ in which ${\Cal P}$ is full-dimensional. Then $e$ is the nearest point of ${\Cal P}$ to $p_0$ and thus to $p$. The converse is clear.
\end{proof}

If $K\leq F$ are faces of ${\Cal P}$ we will write ${\Cal D}(K, F)$ for the codimension of $K$ in $F$, that is, if $F$ has dimension $p$ and $K$ has dimension $q$ then ${\Cal D}(K, F) = p-q$. If $K\leq F$ we will write ${\Cal F}_d(K, F)$ for the set of faces $G$ with $K\leq G\leq F$ and ${\Cal D}(K, G) = d$, and just ${\Cal F}_d(K)$ for the set of all faces $G$ with ${\Cal D}(K, G) = d$.   Note also that $e\in K^\circ$ in the relative topology iff $K$ is the smallest face containing $e$.

\begin{theorem} 
\label{th:qm1p12}
If $F, K$ are faces of ${\Cal P}$, $K<F$, $e\in K^\circ$, and there is an escape from $e$ toward $p\notin {\Cal P}$ along $F$, then there is an escape from $e$ along either K or a face $K^\prime\in {\Cal F}_1(K)$.
\end{theorem}
\begin{proof}
Suppose $e\in K^\circ$ and there is an escape along $F$ for some $F>K$. If ${\Cal D}(K, F) \leq 1$ we are done,  so assume ${\Cal D}(K, F) \geq 2$, and that there is no escape for $K$ or $K^\prime\in {\Cal F}_1(K)$.  We may assume $F$ is minimal of $K$-codimension at least two with the property of having an escape from $e$ along $F$. Let $K< G<F$ and ${\Cal D}(G, F) = 1$. Since $F$ is minimal, there are no escapes from $e$ along $G$ or a subface of $G$ containing $e$, 
so $e$ is the nearest point of $G$ to $p$ by Theorem \ref{th:qm1p11}. 

Since there is an escape along $F$, $e\neq \pi_F$ and there is a point $e^\prime \neq e$ on the line $[e, \pi_F]$ for which the segment $[e, e^\prime]\subset F$. Let $H_e = H\intersect {\Cal A}(F)$ where $H$ is the hyperplane orthogonal to $[e, \pi_F]$ through $e$, that is, $H_e = \{x\in {\Cal A}(F)\mid (x-e)\cdot (\pi_F-e)=0\}$. Let $H^+_e$ be the open half space in ${\Cal A}(F)$ determined by $H_e$ that contains $\pi_F$, so $H_e^+ = \{x\in {\Cal A}(F)\mid (x-e)\cdot (\pi_F-e)>0\}$ and let $H_e^-$ be the corresponding open half space defined for negative inner products. Let ${\tilde B}$ be the open ball centered at $\pi_F$ of radius $\|\pi_F -e\|$, and note that for any $x\in H_e^+$ the segment $[e, x]$ must pass through ${\tilde B}$. 
We claim that for some $\epsilon > 0$ there is an open half-ball (the intersection of  $H^+_e$ with an ordinary ball)  $B_\epsilon$ of radius $\epsilon$ centered at $e$ which does not intersect any facet of $F$. 
Suppose otherwise, so that for any sequence $\{\epsilon_i\}$  tending to zero there is a point  $x_i\in B_{\epsilon_i}\intersect G^\prime$ where $ B_{\epsilon_i}\subset H^+_e $ is an open half-ball centered at $e$ and $G^\prime$ is some facet of $F$, noting that each $ B_{\epsilon_i}$ contains points of $F$. Then some subsequence of $\{x_i\}$ must be contained in a particular facet $G_0 < F$ and converge to $e$. But then $e\in G_0$ and $e$ must be the nearest point of $G_0$ to $\pi_F$, which is a contradiction, since $G_0$ has  points nearer to $\pi_F$ than $e$ is.
So let $B_\epsilon$ be the required half-ball for some $\epsilon$. The set $B_\epsilon$ is connected, contains points of $F$ and does not intersect any facet of $F$, so in fact $B_\epsilon\subset F^\circ$.
We now note that for any facet $G$ of $F$ that contains $e$ we must have ${\Cal A}(G) \subset H_e$. 
If $x\in {\Cal A}(G)\intersect H^+_e$ then for some $g$ on the segment $[e, x]$ we have
$g\in B_\epsilon\intersect {\Cal A}(G)\subset F\intersect {\Cal A}(G) = G$, but then the segment $[e, g]$, hence $G$, contains points nearer to $\pi_F$ than $e$ is, again a contradiction. Similarly there can be no points $y\in {\Cal A}(G)\intersect H^-_e$ since that would force the existence of an $x\in {\Cal A}(G)\intersect H^+_e$. The set $H_e$ has dimension at most $\dim(F) -1$ and therefore exactly $\dim(F)-1 = \dim(G)$, thus ${\Cal A}(G) = H_e$. Since this argument applies for any facet $G$ of $F$ containing $e$ and since $G = F\intersect H_e$, there is only one facet $G<F$ that contains $e$. 
But this is also a contradiction and there can be no such minimal $F$, because  
${\Cal D}(K, F)\geq 2$, and therefore $K$ is an intersection of at least two distinct facets of $F$ (that contain $e$). 
 \end{proof}

By Theorem \ref{th:qm1p12}, the search for an escape (or the pursuit of the lack of one) can be confined to $K$ for which $e\in K^\circ$ or a codimension-$1$ face $F>K$, so for construction of an algorithm we require an enumeration of the codimension-$1$ superfaces of a given $K$. This is easy if the face lattice of ${\Cal P}$ is known in advance, but otherwise may require some work. Such an enumeration might also be taken as part of a scheme for the construction of the face lattice. There will be a trade-off between extra processing to do this, versus a scheme that may admit more containing faces of $K$ but which does include those of codimension-$1$. It is this choice that gives rise to a class of algorithms.

For the construction of an algorithm, we will assume henceforth unless stated otherwise that ${\Cal A}({\Cal P}) = \bbr^n$. From \cite{as-tlip}, if $K$ is a face of ${\Cal P}$ we have ${\Cal A}(K) = \{x\in\bbr^n|A^\prime x = b^\prime\}$ where the system $(A^\prime \mid b^\prime)$ is a subsystem of $(A \mid b)$, and where we use  the latter notation to represent mere sets of rows from the augmented matrix $[A \mid b]$ irrespective of equality or inequality or to represent the augmented matrix itself, depending on context, and where $(A\mid b)$ is the set of rows for the inequalities that define ${\Cal P}$. We will also for convenience confuse a singleton $(a\mid b)$ with a single row of the containing system. 

By the irredundancy assumption each facet of ${\Cal P}$ has the form $\{x\in {\Cal P}\mid a_ix = b_i\} = {\Cal P}\intersect\{x\in\bbr^n \mid a_ix = b_i\}$ for some $(a_i\mid b_i)\in (A\mid b)$.  Each face $K<{\Cal P}$ is an intersection of facets and has the form $K = {\Cal P}\intersect \{x\in \bbr^n\mid A^\prime x = b^\prime\}$ where as noted above  ${\Cal A}(K) =\{x\in \bbr^n\mid A^\prime x = b^\prime\}$ for some set $(A^\prime \mid b^\prime)\subset (A\mid b)$.

For a face $K<{\Cal P}$, we will say that a set $(A^\prime \mid b^\prime)$ {\it represents} $K$ (or ${\Cal A}(K)$) if ${\Cal A}(K) = \{x\in \bbr^n\mid A^\prime x = b^\prime\}$ and will write ${\Cal S}(K)$ for the maximal set $(A^\prime \mid b^\prime)$ that represents $K$. It is easy to show (by an argument similar to that of the claim below) that ${\Cal A}(K)=\{x\in\bbr^n\mid A^\prime x = b^\prime \}$ where $(A^\prime\mid b^\prime) = {\Cal S}(K)$ is the set of all $(a_i\mid b_i)$ such that $a_ix = b_i$ for all $x\in K$.

If $(A^\prime\mid b^\prime)$ represents ${\Cal A}(K)$ then  $\dim(K)=\dim({\Cal A}(K))$ the latter of which is determined by the rank of $A^\prime$. If $\rank(A^\prime)=r$ then the dimension of $K$ (${\Cal A}(K)$) is $n-r$ since that is the dimension of the null space $N(A^\prime)$ which is a translate of ${\Cal A}(K)$.
If $F>K$ and if ${\Cal A}(F) = \{x\in \bbr^n\mid A^\twoprime x = b^\twoprime\}$ then every point of $K$ satisfies all of these equalities so $(A^\twoprime \mid b^\twoprime)\subset {\Cal S}(K)$. 
A point $e$ on the boundary of ${\Cal P}$ will belong to one or more facets of ${\Cal P}$, whose intersection  we will denote as $K_{e} = \{x\in{\Cal P}\mid {A^\prime}x = { b^\prime}\} = {\Cal P}\intersect \{x\in \bbr^n\mid  {A^\prime}x = { b^\prime}\}$ where $( {A^\prime}\mid { b^\prime}) = ( {A^\prime}\mid { b^\prime})_e$ is the set of inequalities satisfied with equality by $e$. For convenience we note the following.

\begin{claim}
We have ${\Cal S}(K_e) = ( { A^\prime}\mid {b^\prime})_e$, ${\Cal A}(K_{e}) = \{x\in \bbr^n\mid  {A^\prime}x = { b^\prime}\}$ and $e\in K_{e}^\circ$.
\end{claim}
\begin{proof}
Let $L=\{x\in \bbr^n\mid  {A^\prime}x = { b^\prime}\}$. We clearly have ${\Cal A}(K_e)\subset L$ from which  $\dim({\Cal A}(K_e))\leq \dim(L)$. On the other hand if ${\Cal A}(K_e)=\{x\in \bbr^n\mid  {A^\twoprime}x = { b^\twoprime}\}$ then $({A^\twoprime}\mid { b^\twoprime})\subset ({A^\prime}\mid { b^\prime})$ since $e$ satisfies the relevant equations. But $\dim({\Cal A}(K_e))=n-r_0$ where $r_0=\rank(A^\twoprime)$  and we must have  $r_0\leq r$ where $r= \rank(A^\prime)$, so $\dim({\Cal A}(K_e))\geq \dim(L)$ and so $L = {\Cal A}(K_e)$. That $e\in K_{e}^\circ$ follows because otherwise we would need to add another equation that $e$ satisfies, and they are already all acounted for.
\end{proof}

\begin{definition}
\label{de:qm1p11}
We will say that ${\Cal E}$ is a {\it codimension-$1$ enumeration (C1E) scheme} for ${\Cal P}$ if for each $K<{\Cal P}$ there is a sequence $\{{\Cal E}_{K, i}\}$ (the {\it enumerator}) of length $|{\Cal E}_K|$ such that for each $i=0, \dots, |{\Cal E}_K|-1$, ${\Cal E}_{K, i}=({\hat A}_i\mid {\hat b}_i)\subset{\Cal S}(K)$, and for each $F\in {\Cal F}_1(K)$ some ${\Cal E}_{K, i}$ represents ${\Cal A}(F)$. 
\end{definition}

 For a face $K_e$ we will denote the enumerations as
$\{{\Cal E}_{e, i}\}$ which will have length $|{\Cal E}_e|$ and each ${\Cal E}_{e, i}=({\hat A}_i\mid {\hat b}_i)\subset{\Cal S}(K_e)$ and require that for each $F\in {\Cal F}_1(K_e)$ some ${\Cal E}_{e, i}$ represents ${\Cal A}(F)$. One possibility for a C1E would be to list every subset of ${\Cal S}(K)$ for each $K$. An {\it efficient} enumerator would  produce a representation for each $F\in {\Cal F}_1(K)$ and nothing more, but this would require more explicit information about ${\Cal P}$, namely the face lattice of ${\Cal P}$. For the computations in this paper we use enumerators that compromise between the former choice and the efficient enumerator.

 If $F\in {\Cal F}_1(K)$ then ${\Cal A}(F)$ is represented by $(A^\twoprime\mid b^\twoprime)$ for some 
$(A^\twoprime\mid b^\twoprime)\subset {\Cal S}(K)$ and it must be true that the rank of $A^\twoprime$ is $r-1$ where $r$ is the rank of $A^\prime$ in $(A^\prime\mid b^\prime) = {\Cal S}(K)$, so we may assume that a representation $ (A^\twoprime\mid b^\twoprime)$ of ${\Cal A}(F)$ has  exactly $r-1$ rows by removing  rows until we have a linearly independent set. 
In other words, the sets $ (A^\twoprime\mid b^\twoprime)$ that represent the affine hulls of  $F\in {\Cal F}_1(K)$ will be found among subsets of ${\Cal S}(K)$ that have $r-1$ elements. This allows the construction of a codimension-$1$ enumerator which we will call the {\it simple enumerator}, by merely picking subsets of size $r-1$ from $(A^\prime\mid b^\prime) = {\Cal S}(K)$ for any $K$, where $r$ is the rank of $A^\prime$. A refinement of this is to select only sets of size $r-1$ that have rank $r-1$. If $r=1$ then ${\Cal F}_1(K)$ is empty and there is nothing to do. 

There are clearly many possibilities for enumerators. For instance, if ${\Cal P}$ is a cone with a single vertex $v$, which happens to be the intersection of a large number $k$ of facets of ${\Cal P}$, then for $e=v$, if the rank $r$ of  $( {A^\prime}\mid { b^\prime})_e$ is relatively small we may end up having $\bigl ( {k\atop r-1}\bigr)$ choices of subsets in the simple enumerator, which may be a somewhat large number. 
A modified enumerator such that if $K=v$ then $\{{\Cal E}_{K, i}\}$ just enumerates the the one-dimensional faces that contain $v$ and which for every other face of ${\Cal P}$ coincides with the simple enumerator, can be used in this case. We note that in all of our examples below, however, we use the unmodified simple enumerator.

In order to present the algorithm, we define a procedure $\hbox{\it{ESC}\hskip 1pt}$ below whose purpose is to supply a potential escape along a given affine subspace. The construction of $\pi_p$ in Definition \ref{de:qm1p2} is the projection of $p$ onto the affine subspace with definition $\{x\in \bbr^n\mid {\hat A}x = {\hat b}\}$ and assumes that ${\hat A}e_0 = {\hat b}$. This projection can be achieved in any convenient way, one possible of which is to use the methods outlined in \cite{ms-lufam}. Once the projection is found, the scheme will follow the direction $\pi_p - e_0$ from $e_0$ to either an exit from ${\Cal P}$ or to $\pi_p$ if $\pi_p$ is reached without exiting ${\Cal P}$. Then the result is either an escape or not an escape, depending upon whether or not movement from $e_0$ within ${\Cal P}$ in the given direction was actually possible. 

As to the projections used in our example computations, we require calculation of the quantities ${A}^+{A}x$ where $x$ is a given vector and $A_{m\times n}$ is an $m \times n$ matrix. If $A$ has rank $r$, then by the results of  \cite{ms-lufam} we can write $PA = LU$ where $P$ is a permutation matrix and $L$ and $U$ represent a rank decomposition of $PA$, with $L_{m\times r}$ and $U_{r\times n}$ both of rank $r$ where $L$ is truncated lower triangular and $U$ is in echelon form with leading nonzero entries equal to  $1$ (these are implicit, so not stored). In this case it is easily shown that $${A}^+{A} = U^*(UU^*)^{-1}U.$$
The factorization is by means of a compact scheme, so that $L$ and $U$ are stored in the same locations as occupied by the original matrix $A$. Moreover we can compute and store the lower triangular part of $UU^*$ in the upper left (lower triangular) $r \times r$  part of $L$ in $PA$ without disturbing the contents of $U$, since for this purpose we no longer need the contents of $L$. Given $x\in \bbr^n$ we compute $y = Ux$ and then $z = (UU^*)^{-1}Ux$ by solving the equation $UU^* z = y$ using well known methods that only require the lower triangular part of the Hermitian positive definite $r \times r$ matrix $UU^*$. Then we have ${A}^+{A}x = U^*z$ as the required projection. 
\begin{definition}
\label{de:qm1p2}
Let $({\hat A}\mid {\hat b}) \subset (A\mid b)$ where as usual $(A\mid b)$ is that defining ${\Cal P}$, and let $e_0\in {\Cal P}$ satisfy ${\hat A}e_0 = {\hat b}$. Define the procedure $\hbox{\it{ESC}\hskip 1pt}$ by
\end{definition}
{
\setlength{\jot}{3pt}
\begin{align}
\hbox{\it{ESC}\hskip 1pt}&({\hat A}, e_0)\qu{.5}\{\notag\\
&\hbox{set}\qu{1.4} \pi_p \leftarrow  p - {\hat A}^+{\hat A}(p-e_0),\notag\\
&\qu{2.6}t_0 \leftarrow \max \{t\in [0,1]\mid e_0 + t(\pi_p -e_0) \in {\Cal P}\}, \notag\\
&\qu{2.6}e^\prime_0 \leftarrow  e_0 + t_0(\pi_p -e_0). \notag\\
&\hbox{return}\qu{1}e^\prime_0\notag\\
\}\qu{1.5}&\notag
\end{align}
}
\vskip 15pt
\noindent 

For the following algorithm, the initial construction determines the set of inequalities satisfied with equality by $e_0$, that is ${\Cal S}(K_{e_0}) = ({\hat A}\mid {\hat b})$, where $e_0\in K_{e_0}^\circ$, where the point $e_0$ is some arbitrary $p$-visible point.
\vskip 10pt
\begin{algorithm}
\label{al:qm1p1}
Let $p\notin {\Cal P}$ and $e_0 \in {\Cal P}$ be $p$-visible. Let ${\Cal E}$ be a codimension-$1$ enumeration scheme for ${\Cal P}$ and set $({\hat A}\mid {\hat b}) \leftarrow {\Cal S}(K_{e_0})$. 
\end{algorithm}
\vskip -5pt
{
\setlength{\jot}{3pt}
\begin{align}
\hbox{set}\qu{.5}&\hbox{\it{done}} \leftarrow \hbox{false}. \notag\\
\hbox{whil}&\hbox{e}\hbox{ not} \qu{.25}  \hbox{\it{done}} \qu{.5}\{ \notag\\
&\hbox{set}\qu{.5} \hbox{\it{done}} \leftarrow \hbox{true}, \qu{.5}e^\prime_0 \leftarrow  \hbox{\it{ESC}\hskip 1pt}({\hat A}, e_0).\label{eq:qm1p2}\\
&\hbox{if}\qu{.5} e_0^\prime \neq e_0 \qu{.5}\{ \notag\\
&\qu{2}\hbox{set}\qu{.5} ({\hat A}\mid {\hat b}) \leftarrow {\Cal S}(K_{e^\prime_0}), \label{eq:qm1p3}\\
&\qu{2}e_0 \leftarrow e_0^\prime,\qu{.5} \hbox{\it{done}} \leftarrow \hbox{false}.\notag\\
&\}\notag \\ 
&\hbox{else if}\qu{.5}\rank({\hat A})>1\qu{.5}\{ \notag\\
&\qu{2}\hbox{for each} \qu{.5} i=0,\dots,|{\Cal E}_{e_0}|-1 \qu{.5}\hbox{while}\qu{.5}\hbox{\it{done}}\qu{.5}\{ \notag\\
&\qu{4}\hbox{set}\qu{.5} ({\bar A}\mid {\bar b}) \leftarrow {\Cal E}_{e_0,i},\label{eq:qm1p4}\\
&\qu{4}e^\prime_0 \leftarrow  \hbox{\it{ESC}\hskip 1pt}({\bar A}, e_0). \label{eq:qm1p5}\\
&\qu{4}\hbox{if}\qu{1} e_0^\prime \neq e_0 \qu{1}\{ \notag\\
&\qu{6}\hbox{set}\qu{.5} ({\hat A}\mid {\hat b}) \leftarrow {\Cal S}(K_{e^\prime_0}), \qu{1} \label{eq:qm1p6}\\
&\qu{6}e_0 \leftarrow e_0^\prime, \qu{.5} \hbox{\it{done}} \leftarrow \hbox{false}.\notag\\
&\qu{4}\}\notag\\
&\qu{2}\}\notag\\
&\}\notag\\
\}\qu{1.4}&\notag\\
\notag
\end{align}
}
\vskip -5pt

If the scheme of Algorithm \ref{al:qm1p1} fails to find an escape within $K_{e_0}$ (as in (\ref{eq:qm1p3})), then it becomes necessary to try faces containing $K_{e_0}$.  The loop containing (\ref{eq:qm1p4}) eventually (applying an enumerator) examines each of the codimension-$1$ spaces containing ${\Cal A}(K_{e_0})$ for the existence of an escape, which is all that is necessary by Theorem \ref{th:qm1p12}. If none is found in either (\ref{eq:qm1p2}) or (\ref{eq:qm1p4}) then the current value of $e_0$ is the point of ${\Cal P}$ minimizing the distance to $p$.

It should be noted that Algorithm \ref{al:qm1p1} will also work if the set of constraints for ${\Cal P}$ is not irredundant. If we view the constraints on ${\Cal P}$ as a set $I\union S$ where $I$ is a set of irredundant constraints and $S$ are superfluous constraints, then any scheme for which the enumerator picks at least what it would pick if $I$ were the only set of constraints (such as one that picks sets of size $r-1$, or sets of size $r-1$ of full rank), will work but be less efficient.\par
\medskip

\noindent\textit{Remark 1.} 
An enumerator is only required to select appropriate subsets of some ${\Cal S}(K_{e_0})$. For an enumerator that just selects subsets of $({\hat A}\mid {\hat b}) = {\Cal S}(K_{e^\prime_0})$ of size $r-1$  where $\rank({\hat A}) = r$  (as in (\ref{eq:qm1p3})), not all selections will necessarily have rank $r-1$, and even for those that do, there is no guarantee that the corresponding affine subspace in (\ref{eq:qm1p4}) corresponds to a codimension-$1$ face $F>K_{e^\prime_0}$, but in that case there will be no escape.  It is also possible that an escape will be found in (\ref{eq:qm1p5}) along a face that has codimension greater than 1. This is largely immaterial, again because eventually either the loop (\ref{eq:qm1p4}) -- (\ref{eq:qm1p6}) will be fully exhausted or (\ref{eq:qm1p3}) represents a facet of ${\Cal P}$ in which case the algorithm will stop there. The trade-off for this enumerator is that there is extra computation for the subsets that don't bear fruit, but the extra computation of ensuring that every selection from $({\hat A}\mid {\hat b})$ will represent a codimension-$1$ face is avoided, or one can take the point of view that this {\it is} the extra computation. 

\par
\medskip
Let $\{e^{(i)}\}$ be the sequence of escapes produced by Algorithm \ref{al:qm1p1} for some starting point $e^{(0)} = e_0$. We will say that this sequence {\it descends} into a face $K<{\Cal P}$ if ${\Cal A}(K) = \{x\in \bbr^n\mid {\hat A}x = {\hat b}\}$ where $({\hat A}\mid {\hat b})$ is the system in (\ref{eq:qm1p3}) and ${\hat A}$ in (\ref{eq:qm1p3}) is distinct from that in (\ref{eq:qm1p2}). This will typically be the case when the escape along a face is through a subface on the boundary of ${\Cal P}$, in which the system of (\ref{eq:qm1p3}) consists of the system in (\ref{eq:qm1p2}) with additional equations. We will say that the sequence {\it ascends} through $K$ if the system for $K$ is that in (\ref{eq:qm1p2}) and no escape is found in (\ref{eq:qm1p3}) but an escape is found at (\ref{eq:qm1p5}) inside the loop along some face of codimension $\geq 1$ containing $K$. If the sequence neither ascends nor descends, then (\ref{eq:qm1p2}), (\ref{eq:qm1p3}) simply moves the point to the projection onto the face whose system is (\ref{eq:qm1p2}), and this will occur just once, followed by an ascent, or the algorithm terminates. This follows because according to the definition, an escape from  $e\in K$ to $e^\prime \in K^\circ$ will always result in $e^\prime = \pi_K$. Next we show convergence of Algorithm \ref{al:qm1p1}, at least in the sense of termination of the sequence of iterates.
\begin{theorem} 
\label{th:qm1p2}
Algorithm \ref{al:qm1p1} completes in at most finitely many steps.
\end{theorem}
\begin{proof}
Let $\{e^{(i)}\}$ be the sequence generated by the algorithm, and note first that the sequence $\{\|e^{(i)}-p\|\}$ is strictly decreasing by construction. Note also that if the sequence ascends through $F$, with say $e^{(i)}\in F$ and $e^{(i+1)} \in G$ where $F$ is a proper face of $G$ (${\Cal A}(G)$ represented by ${\Cal E}_{e_0, i}$) as in (\ref{eq:qm1p5}), then we have $e^{(j)}\notin F$ for any $j>i$, that is, the sequence never revisits $F$. This is because the loop containing (\ref{eq:qm1p4}) is not entered unless there is no escape within $F$, and if not, then $e^{(i)}\in F^\circ$ is in fact $\pi_F$, and minimizes the distance to $p$ over ${\Cal A}(F)$. A subsequent $e^{(j)}\in F$ would improve upon that distance, which is not possible.

Now we proceed by induction on the codimensions of faces of ${\Cal P}$, defining $\Gamma_k$ as the set of faces $F<{\Cal P}$ of codimension-$k$ in ${\Cal P}$, so that e.g., $\Gamma_1$ is the set of facets of ${\Cal P}$. The non-revisitation of faces is the essence of the story, but we will place an explicit upper bound $i_F$ on the number of elements of the sequence contained in $F^\circ$.
If $F\in \Gamma_1$ then $F^\circ$ might possibly contain the initial point $e^{(0)}$, and if $\pi_F\in F^\circ$ then $\pi_F$ will belong to the sequence of escapes if there is any other $e^{(i)}\in F^\circ$. 

In general we note that a descent moves from a point in some $F^\circ$ to some $K^\circ$ where $K$ is a proper face of $F$, since we always start at a point $e_0$ with the system $({\hat A}\mid {\hat b})$ being that satisfied by $e_0$, so that $e_0\in K_{e_0}^\circ$ where $({\hat A}\mid {\hat b}) = {\Cal S}(K_{e_0})$. On the other hand an ascent into $F$ moves from a point interior to a proper face $F_0<F$ into $K^\circ$ where $K$ is potentially any face of $F$ not a subface of 
$F_0$, 
with ${\Cal A}(K)$ as determined in (\ref{eq:qm1p6}).  If $F\in \Gamma_1$ the sequence will not enter $F^\circ$ by descent, and if an ascent results in an element $e^{(i)}\in F^\circ$ then in fact $e^{(i)} = \pi_F$. Although there are potentially $f_F$ ascents into $F^\circ$ where $f_F$ is the number of proper faces of $F$,  any of those will result in $\pi_F$, which has already been accounted for. Since there is no other way for the sequence to enter $F^\circ$, if $F\in \Gamma_1$ we have $i_F \leq 2$.

Let $F\in\Gamma_k$ for some $k>1$ and assume that $G^\circ$ contains at most finitely many elements $i_G$ of the sequence for all $G\in \Gamma_j$ for all $j<k$. As usual, the sequence can originate in $F$, and if $F^\circ$ happens to contain $\pi_F$ that point will be included in the sequence (if  $F^\circ$ has any elements of the sequence) which will account for the usual possible $2$ elements. The sequence $\{e^{(i)}\}$ can enter $F^\circ$ through descent from some $G^\circ$ where $G>F$ and otherwise through an ascent from a proper face of $G$ other than  one containing $F$. Thus for any $G>F$ we will get no more than $f_G-\sigma_{F,G}+ i_G-1$ elements of $F^\circ$ by ascent into $G$ or descent from $G^\circ$, where the number by descent excludes $\pi_G$ because that point will never descend to a proper face of $G$, and  ascent from a proper face of $G$ containing $F$ is also excluded, where $\sigma_{F,G} = |\{G_0<G\mid F\leq G_0\}|$. We can get elements from below by ascent from a proper face of $F$ into $F^\circ$ but as we have noted, any of these will result in $\pi_F$, which has been accounted for. Therefore we have the bound
\begin{equation} i_F \leq 2 +\sum_{j<k, G\in \Gamma_j\atop {F<G}}(f_G-\sigma_{F,G}+ i_G-1)\label{eq:qm1p7} \end{equation}
on the number of elements of the sequence in $F^\circ$, so inductively   $i_F$ is finite for each $F\in \Gamma_k$ for each $k = 1, \ldots, n = \hbox{Dim}({\Cal P})$, and of course for $F\in\Gamma_n$ we have the better bound $i_F = 1$. In any case since ${\Cal P}$ has finitely many faces, this shows that $\{e^{(i)}\}$ is finite.
\end{proof}

\vskip 5pt
\noindent\textit{Remark 2.} The starting point for Algorithm \ref{al:qm1p1} can be any $p$-visible point of ${\Cal P}$, so we can get a starting point $e_0$ from any point of $x\in{\Cal P}$ by taking $e_0$ as the only $p$-visible point on the segment $[x, p]$. Starting at a vertex of ${\Cal P}$ is less desirable, since it immediately forces an ascent unless the vertex is actually the solution.
\section{Experiments}
In this section we demonstrate some properties of Algorithm \ref{al:qm1p1} with  simple low-dimensional examples and then some experimental trials in higher dimensions. These methods are easily scaled for real-world problems.

For the singular problems below, we require a brief digression into the construction of constraints on the linear image of a polyhedron. Let ${\Cal P}$ be a polyhedron and let $\Lambda:\bbr^m\rightarrow\bbr^n$ be a linear transformation with restriction $\Lambda:{\Cal P}\rightarrow \Lambda {\Cal P}$. We will confuse $\Lambda$ and its matrix representation $\Lambda_{n\times m}$ as required. Let ${\Cal P}$ be defined by constraints ${\Cal P}=\{x\in\bbr^m\mid Bx\leq c\}$ for some matrix $B$ and vector $c$. If $\Lambda$ is invertible, one can easily construct constraints on $\Lambda {\Cal P}$ as $\{x\in\bbr^n\mid B\Lambda^{-1}x\leq c\}$. If $\Lambda$ is not invertible, the construction is not quite as straightforward. We will show that constraints can be constructed with varying degrees of `easily', i.e., computational complexity, the indicator of this being substantially more a function of the right corank $\rho_m(\Lambda)$ of $\Lambda$ than the values of $m$ or $n$, where $\Lambda$ is $n\times m$ and where this quantity is defined as $\rho_m(\Lambda) = m-r$ where $r$ is the rank of $\Lambda$, and is the dimension of the null space of $\Lambda$. We will also use the idea of the usual (left) corank $\rho_n(\Lambda) = n-r$. 

Let $\Lambda = P^*LU$ where $L_{n\times r}$ and $U_{r\times m}$ are truncated lower triangular and upper echelon form matrices of rank $r$ where $\Lambda$ has rank $r$ and where $P$ is a permutation matrix, as in \cite{ms-lufam}. The leading nonzero entry in each row of $U$ is $1$ and in particular there is an integer array $\gamma$ of length $r$ such that 
$$U_{ij} = \begin{cases}1&\hbox{ if } j = \gamma_i,\cr 0&\hbox{ if }j<\gamma_i,\cr u_{ij}&\hbox{ if }j>\gamma_i\end{cases}$$ for each $i=0,\ldots, r-1$, $j=0,\ldots,m-1$ with elements $u_{ij}$  determined by the factorization of $\Lambda$. For the transformation $U:{\Cal P}\rightarrow U{\Cal P}={\Cal P}^\prime$ a preimage $x$ of $y\in {\Cal P}^\prime$ can be constructed by back substitution, with
\begin{equation}\quad x_{\gamma_i} \leftarrow y_i -\sum_{m-1\geq k>\gamma_i} U_{ik}x_k \label{eq:qm2p0}\end{equation} 
for $r-1\geq i\geq 0$ in that order of $i$. The calculation (\ref{eq:qm2p0}) does not determine $x_j$ for $j\notin \gamma$. These quantities can be selected arbitrarily, and taken together with the components determined by (\ref{eq:qm2p0}) define a preimage $x\in \bbr^m$ of $y$ under $U$. Let $\nu$ be the complement (if any) of the set $\gamma$ in the set of integers $0,\ldots,m-1$, and define the {\it suspension} of $U$ as the matrix ${\tilde U}$ given by $${\tilde U}_{ij} = \begin{cases}U_{kj}&\hbox{ if } i=\gamma_k\in\gamma\cr
\delta_{ij}&\hbox{ if }i\in\nu\cr\end{cases}$$ for $i,j=0,\ldots,m-1$ where $\delta$ is the Kronecker delta. Note that the suspension ${\tilde U}:\bbr^m\rightarrow \bbr^m$ and that ${\tilde U}$ is invertible. For $x\in\bbr^m$ we can write ${\tilde U}x = {\tilde x}$ where \begin{equation} {\tilde x}_i = \begin{cases}y_k&\hbox{ if }i=\gamma_k\in\gamma,\cr x_i&\hbox{ if }i\in\nu\cr\end{cases}\label{eq:qm2p1}\end{equation} where $y_k$ are components of $Ux$. The constraints on ${\tilde U}{\Cal P} = {\Cal P}^\twoprime$ are of the form 
\begin{equation} B^\prime {\tilde x}\leq c \label{eq:qm2p2} \end{equation} 
where $B^\prime  = B{\tilde U}^{-1}$. There are clearly $\rho_m(\Lambda)$ free parameters in any preimage $x$ of $y\in U{\Cal P}$ and these are unaffected under transformation by ${\tilde U}$, simply being reproduced in the components $x_i$ for $i\in\nu$. Any of the constraints of (\ref{eq:qm2p2}) can be rearranged into an expression involving the free parameters $x_i$ for $i\in \nu$ and the components $y_k$ of $y\in U{\Cal P}$, and the columns of $B^\prime$ can be (or be imagined to be) arranged so that the free parameters come first. Doing this, we apply just $\rho_m(\Lambda)$ steps of Fourier-Motzkin  elimination (aka FME, see \cite{as-tlip}) to eliminate these free parameters, to obtain `corank reduced constraints' of the form $B^\twoprime y\leq c^\prime$ involving only the components of $y$, or no constraints at all (in the case that $B^\twoprime y\leq c^\prime$ is an empty set of inequalities, satisfied by any $y\in \bbr^r$) if at some point of the FME reduction all coefficients of one of the free variables are strictly positive or strictly negative. While the bad news is that Fourier-Motzkin is a doubly exponential time algorithm, the good news is that we can ignore the bad news if the corank is sufficiently small. If $\rho_m(\Lambda) = 0$ we take (\ref{eq:qm2p2}) as the definition of $B^\twoprime y\leq c^\prime$, applying no steps of FME. The following argument applies for the reduced constraints, whether empty or not. 
\begin{theorem}
\label{th:qm2p1}
The corank reduced constraints constitute a complete set of constraints for $U{\Cal P}$. \end{theorem}
\begin{proof}
If $\rho_m(\Lambda) = 0$ there is nothing to show, so assume $\rho_m(\Lambda) > 0$. Let $y\in U{\Cal P}$ and let $x\in {\Cal P}$ be a preimage of $y$, constructed as in (\ref{eq:qm2p0}) for some choice of free parameters. Then ${\tilde x} = {\tilde U}x$ satisfies the constraints (\ref{eq:qm2p2}). But then $y$ satisfies the reduced constraints $B^\twoprime y\leq c^\prime$ since FME produces equivalent sets of constraints at each step. On the other hand, suppose $y\in \bbr^r$ satisfies the reduced constraints. If $y\notin U{\Cal P}$ then no preimage $x\in\bbr^n$ of $y$ is an element of ${\Cal P}$, so ${\tilde x}$ violates at least one of the constraints (\ref{eq:qm2p2}). But since $y$ satisfies the reduced constraints produced by FME, it is possible to work backward to find values of the eliminated variables for which we do have a solution of (\ref{eq:qm2p2}). This is a contradiction, so in fact $y\in U{\Cal P}$.
\end{proof}
To complete the construction we define constraints $B^\threeprime z\leq c^\twoprime$ as the set of constraints $B^\twoprime(L^*L)^{-1}L^*Pz\leq c^\prime$ (if $B^\twoprime y\leq c^\prime$ is nonempty) together with the auxilliary constraints $\zeta_i^*z\leq 0, \qu{0.5} -\zeta_i^*z\leq 0$  where $\{\zeta_i\}_{i=1}^{\rho_n(\Lambda)}$ is a basis of the null space of $L^*P$. 
\begin{theorem} 
\label{th:qm2p2}
The constraints $B^\threeprime z\leq c^\twoprime$ are a complete set of constraints for $\Lambda{\Cal P}$.\end{theorem}
\begin{proof} If $B^\twoprime y\leq c^\prime$ is nonempty and $z$ satisfies $B^\twoprime(L^*L)^{-1}L^*Pz\leq c^\prime$ then we have $(L^*L)^{-1}L^*Pz=Ux$ for some $x\in {\Cal P}$ by Theorem \ref{th:qm2p1}. If $z$ also satisfies the auxilliary constraints then $z = P^*Lw$ for some  $w\in \bbr^r$ from which $z = P^*LUx\in \Lambda{\Cal P}$. If the reduced constraints are empty then $U{\Cal P} = \bbr^r$, in which case $\Lambda{\Cal P}$ is simply the range of $P^*L$ and we still have $z = P^*LUx\in \Lambda{\Cal P}$ for some $x\in {\Cal P}$. Finally, it is clear that points of $\Lambda{\Cal P}$ satisfy the constraints $B^\threeprime z\leq c^\twoprime$. \end{proof}

\subsection{Constrained Least Squares}
The scheme of  Algorithm \ref{al:qm1p1} is applied to simple examples of constrained least squares, one of which is nonsingular, the others singular. To this end we use pseudorandomly generated sample data which is presented in Table \ref{ta:qm8} of Appendix 1. Our examples minimize the quantity \begin{equation}\|Ax -b\| \label{eq:qm2p3}\end{equation} subject to $x\in {\Cal P}$ where ${\Cal P}$ is the polyhedron $\{x\in\bbr^4\mid x_i\leq 2 \hbox{ for each } i\}$. 

\subsubsection{Nonsingular case}
For the nonsingular case, we minimize (\ref{eq:qm2p3}) where $A$ and $b$ are the matrix and data of Table \ref{ta:qm8}. Minimization of (\ref{eq:qm2p3}) is equivalent to minimizing $\|y-p\|$ for $y\in {\Cal P}^\prime = K^*{\Cal P}$ where $p = K^{-1}A^*b$ where $K$ is the lower triangular Choleski factor of $A^*A$ that is, $A^*A = KK^*$, since we easily have for $y= K^*x$ that \begin{equation}\|Ax-b\|^2 = \|y-p\|^2+\|b\|^2 -\|p\|^2, \label{eq:qm2p4}\end{equation}
which is a reduction possible for many quadratic minimization problems. For instance, it is possible for an objective function of the form ${1\over 2}x^*Qx+c^*x$ if $Q$ is hermitian positive semidefinite and $c$ is in the range of $Q$.

Expressing the constraints of ${\Cal P}$ as $Bx\leq c$, then the constraints on ${\Cal P}^\prime$ are $B^\prime y\leq c$ where $B^\prime = B{K^*}^{-1}$ (inverse matrices are not actually used here and the appropriate linear equations are solved instead). The scheme completes in three steps, with iterates (in $y$ space) depicted in Table \ref{ta:qm1}. The first row of the table depicts the starting point, obtained by following a line from the origin to the first $p$-visible point on ${\Cal P}$ in the direction of $p$ . The last row of the table corresponds to the step consisting of the loop (\ref{eq:qm1p4}), which is effectively the process of deciding that the algorithm is done.
\renewcommand{\arraystretch}{1.2}
\begin{center}
\begin{table}[h!]
\refstepcounter{table}\label{ta:qm1}
\fontsize{8}{10}\selectfont
\centering
{\setlength{\tabcolsep}{.7em}
\begin{tabular}{|c|c|c|c|c|}
 \hline 
 \multicolumn{4}{|c|}{\textbf{Iterate}} & \multicolumn{1}{c|}{\textbf{Distance}} \\ \hline
9.0100\tten{1} & 1.1390\tten{2} & 3.7820\tten{1} & 5.6710\tten{0} & 69.7510 \\ \hline
1.2381\tten{2} & 1.1317\tten{2} & 4.8995\tten{1} & 6.0003\tten{0} & 54.7219 \\ \hline
1.2422\tten{2}  & 1.1293\tten{2} &  5.1272\tten{1} &  8.0725\tten{0} & 54.6331 \\ \hline
1.2422\tten{2}  & 1.1293\tten{2} &  5.1272\tten{1} &  8.0725\tten{0} & 54.6331 \\ \hline
\end{tabular}}
\vskip 3pt
\end{table}
\end{center}
\tcap{\ref{ta:qm1}.}{Iterates for the Nonsingular Case}
The final iterate is converted back to a point in ${\Cal P}$ by multiplication by ${K^*}^{-1}$, with results shown in Table \ref{ta:qm2}, where the constrained result is compared to the free space unconstrained solution.
\vfill \break
\renewcommand{\arraystretch}{1.1}
\begin{center}
\begin{table}[h!]
\refstepcounter{table}\label{ta:qm2}
\fontsize{8}{10}\selectfont
\centering
{\setlength{\tabcolsep}{1em}
\begin{tabular}{|c|c|c|c|c|}
 \hline 
 \textbf{Constraint}&\multicolumn{4}{c|}{\textbf{Solution}}  \\ \hline
${\Cal P}$ & 2.0000 & 2.0000 & 0.7122 & \hphantom{-}0.1321  \\ \hline
None  & 2.9756 & 2.4386 & 1.3741 & -0.2178   \\ \hline
\end{tabular}}
\end{table}
\end{center}
\vskip 5pt
\tcap{\ref{ta:qm2}.}{Solutions for the Nonsingular Case}
\subsubsection{Singular cases}
{\it Singular Case 1. }For the first singular case, we minimize (\ref{eq:qm2p3}) where $A$ and $b$ are merely the first three rows of the matrix and data of Table \ref{ta:qm8} in the appendix, subject to $x\in {\Cal P}$. For this and other singular cases, we drop dimension to the rank of $A$ and solve the usual minimization problem. For this purpose we factor the matrix $A_{m\times n}$ as $A= P^*LU$ where $P$ is a permutation matrix and $L$ and $U$ are matrices of rank $r$ and size $m\times r$ and $r\times n$ respectively,  where $r$ is the rank of $A$, as in \cite{ms-lufam}. 

The conversion still has the form of (\ref{eq:qm2p4}) where now $y = K^*Ux$ and $p = K^{-1}L^*Pb$, and $K$ is now the lower triangular part of the Choleski factorization of  $L^*L$ (in this particular example $L$ is square, but will not be in the general  case). The minimization takes place as usual, but in the lower dimension $r$. In this case, it is necessary to determine the constraints on the polyhedron ${\Cal P}^\prime = U{\Cal P}$, for which we apply Theorem \ref{th:qm2p1}, since we will only need these constraints and not those for $A{\Cal P}$. The corank of $A$ is one, and we remove just one free variable by one step of FME. We have, presenting everything to four places,
\begin{equation}U=\left[\begin{matrix} 1&0.8706&-0.2811&-0.9736\cr
0&1&-0.4255&-1.1509\cr 0&0&1&-0.9172\cr \end{matrix}\right] \label{eq:qm2p5}\end{equation} and
\begin{equation}L=\left[\begin{matrix} -14.6785&0&0\cr
16.8958&-19.7627&0\cr
 -1.0007&-1.9354&-6.9264\cr
 \end{matrix}\right]. \label{eq:qm2p6}\end{equation}
 \vskip 5pt
\noindent The corank reduced constraints on $U{\Cal P} = {\Cal P}^\prime$  are obtained as
\begin{equation}B^\twoprime z\leq c^\prime:\quad\begin{aligned} 16.9735z_0-13.5625z_1-z_2&\leq 36.3762\cr 
32.3847z_0-28.1941z_1+z_2&\leq 72.5577\cr
11.1899z_0-\qu{0.5}9.7419z_1-z_2&\leq 24.8479\cr
\end{aligned} \label{eq:qm2p7} \end{equation}
and the minimization takes place in ${\Cal P}^\twoprime = K^*{\Cal P}^\prime$ with constraints ${\widetilde B}^\threeprime \leq c^\prime$ on $y$ where ${\widetilde B}^\threeprime = B^\twoprime {K^*}^{-1}$. The scheme again completes in three steps. The iterates in ${\Cal P}^\twoprime$ are depicted in Table \ref{ta:qm3}. 
\vfill \break
\renewcommand{\arraystretch}{1.2}
\begin{center}
\begin{table}[h!]
\refstepcounter{table}\label{ta:qm3}
\fontsize{8}{10}\selectfont
\centering
{\setlength{\tabcolsep}{.7em}
\begin{tabular}{|c|c|c|c|}
 \hline 
 \multicolumn{3}{|c|}{\textbf{Iterate}} & \multicolumn{1}{c|}{\textbf{Distance}} \\ \hline
5.0810\tten{1} & 5.6320\tten{0} & 7.6560\tten{0} & 83.3315 \\ \hline
5.1664\tten{1} & 7.6226\tten{0} & 9.6801\tten{0} & 82.0161 \\ \hline
6.0538\tten{1}  & 3.6588\tten{1} &  2.0629\tten{1} &  75.4254 \\ \hline
6.0538\tten{1}  & 3.6588\tten{1} &  2.0629\tten{1} &  75.4254 \\ \hline

\end{tabular}}
\end{table}
\end{center}
\tcap{\ref{ta:qm3}.}{Iterates for the First Singular Case}
Next we recover $z = Ux = {K^*}^{-1}y$ as
\begin{equation}z = \left[\begin{matrix}4.2837 \cr 2.4548 \cr 3.0407 \cr\end{matrix}\right] \label{eq:qm2p8}\end{equation}
and finally recover $x$ by back substitution as in \cite{ms-lufam} noting that $x_3$ is a free parameter, obtaining
\begin{equation}x = \left[\begin{matrix}1.8749-0.1103 x_3 \cr 3.7487 +1.5412 x_3 \cr 3.0409 +0.9172 x_3 \cr x_3 \cr\end{matrix}\right]. \label{eq:qm2p9}\end{equation}
For a problem with potentially more free parameters, after the back substitution phase we obtain a vector such as (\ref{eq:qm2p9}) that must satisfy the constraints in the original problem space. This will be a  polyhedron in $\bbr^{\rho_m(A)}$ that is the result of applying the original problem constraints to a vector such as (\ref{eq:qm2p9}). For the present case this polyhedron is a vertex and there is exactly one value of $x_3$ for which the vector in (\ref{eq:qm2p9}) satisfies the constraints of ${\Cal P}$, namely $x_3 = -1.13466536$ to 8 places. The resulting solution is depicted in Table \ref{ta:qm4}, again compared with the unconstrained singular problem.
\renewcommand{\arraystretch}{1.1}
\begin{center}
\begin{table}[h!]
\refstepcounter{table}\label{ta:qm4}
\label{ta:qm4}
\fontsize{8}{10}\selectfont
\centering
{\setlength{\tabcolsep}{1em}
\begin{tabular}{|c|c|c|c|c|}
 \hline 
 \textbf{Constraint}&\multicolumn{4}{c|}{\textbf{Solution}}  \\ \hline
${\Cal P}$ & 2.0000 & 2.0000 & 2.0000 & -1.1347  \\ \hline
None  & 5.5883 & 0.1383 & 1.8019 & -1.2494   \\ \hline
\end{tabular}}
\end{table}
\end{center}
\vskip 3 pt
\tcap{\ref{ta:qm4}.}{Solutions for the First Singular Case}
\noindent
{\it Singular Case 2. }The second singular case is less an example application of Algorithm \ref{al:qm1p1} than an application of Theorem \ref{th:qm2p1}. For the second case, we reduce the rank of the problem to $2$ similarly to the first case, by taking just the first two rows of $A$ and $b$ of Table \ref{ta:qm8} in the appendix. In this case, applying Theorem \ref{th:qm2p1} results in $U{\Cal P} = \bbr^2$, i.e., the image of $U$ is unconstrained. The solution is that of the unconstrained minimization of $\|K^*z-p\|$ where $p$ is that in Case 1. But $K$ is invertible so we have $z={K^*}^{-1}p = (L^*L)^{-1}L^*Pb$. To four places we have

\begin{equation}U=\left[\begin{matrix} 1&0.8706&-0.2811&-0.9736\cr
0&1&\qu{0.75}3.1533&-4.4333\cr \end{matrix}\right]\label{eq:qm2p10}\end{equation} and
\begin{equation}L=\left[\begin{matrix} -14.6785&0\cr
-1.0007&-1.9354\cr
 \end{matrix}\right].\label{eq:qm2p11}\end{equation}
\noindent
 \vskip 5pt\noindent
In this case we have 

\begin{equation}z = \left[\begin{matrix}6.4186 \cr 11.3591\end{matrix}\right] \label{eq:qm2p12}\end{equation}
and we recover $x$ by back substitution obtaining

\begin{equation}x = \left[\begin{matrix}
-3.4706 +3.0263x_2 -2.8860x_3 \cr
 11.3591 -3.1533x_2 +4.4333x_3\cr 
 x_2 \cr x_3 \cr\end{matrix}\right].\label{eq:qm2p13}\end{equation}
 \vskip 10pt\noindent
The solution polyhedron for values of (\ref{eq:qm2p13}) that satisfy the original constraints is given by

\begin{equation}\begin{aligned} 
x_2-0.9536x_3&\leq 1.8077\cr  
-x_2+1.4059x_3&\leq -2.9681\cr 
\end{aligned}\label{eq:qm2p14} \end{equation}
\vskip 5pt
The constrained solution corresponding to (\ref{eq:qm2p13}) for the single vertex of the polyhedron (\ref{eq:qm2p14}) is presented in Table \ref{ta:qm5} compared to the unconstrained solution. It is the point of the polyhedron (\ref{eq:qm2p14}) that minimizes the distance from a constrained solution (\ref{eq:qm2p13}) to the unconstrained solution, by an application of Algorithm \ref{al:qm1p1}.

\renewcommand{\arraystretch}{1.1}
\begin{center}
\begin{table}[h!]
\refstepcounter{table}\label{ta:qm5}
\fontsize{8}{10}\selectfont
\centering
{\setlength{\tabcolsep}{1em}
\begin{tabular}{|c|c|c|c|c|}
 \hline 
 \textbf{Constraint}&\multicolumn{4}{c|}{\textbf{Solution}}  \\ \hline
${\Cal P}$ & 2.0000 & 2.0000 & -0.6388
 & -2.5655 \\ \hline
None  & 2.2115 & 1.9857 & -0.4311 & -2.4210   \\ \hline
\end{tabular}}
\end{table}
\end{center}
\tcap{\ref{ta:qm5}.}{Solutions for the Second Singular Case}
\subsection{Machine Trials}

To produce a rough sense of the behavior of Algorithm \ref{al:qm1p1}, we apply it to two polytopes in a variety of dimensions, namely the unit cube $\{x\in \bbr^n\mid 0\leq x_i\leq 1 \qu{0.5}\hbox{for}\qu{0.5} i=0,\ldots, n-1\}$ and the simplex $\{x\in \bbr^n\mid x_i\geq 0, \qu{0.5} \hbox{for}\qu{0.5} i=0,\ldots, n-1\qu{0.5}\hbox{and} \qu{0.5} \sum_{i=0}^{n-1}x_i \leq 1\}$. These have comparable numbers of constraints, but the cube has quite a few more vertices in higher dimensions. If the starting point is chosen properly, the algorithm has comparable behavior in both cases. Execution time will of course increase in higher dimensions, since more storage and processing is involved for just the constraints.
\vfill\break
\renewcommand{\arraystretch}{1.1}
\begin{center}
\begin{table}[h!]
\refstepcounter{table}\label{ta:qm6}
\fontsize{8}{10}\selectfont
\centering
{\setlength{\tabcolsep}{.7em}
\begin{tabular}{|l|c|r|r|r|r|}
\hline
\bf Polyhedron & \bf Dimension  & \bf Steps & \bf MSec  & \bf Max MSec   & \bf Ascent \\ \hline
Cube       & 10  & 9    & 0.7   & 2 & 0                      \\ \hline
           & 20  & 16    & 4.7  & 10 & 0                      \\ \hline
           & 50 & 31   & 60.3 & 161 & 0                      \\ \hline
Simplex    & 10  & 11    & 0.7   & 5    & 0                   \\ \hline
           & 20  & 23    & 5.7  & 15   & 0                    \\ \hline
           & 50 & 45   & 67.0 & 394   & 0                    \\ \hline
\end{tabular}}
\end{table}
\end{center}
\tcap{\ref{ta:qm6}.}{Barycenter}
\vskip -5pt
The results are depicted in two tables, where Steps represents the average number of calls to $\hbox{\it{ESC}\hskip 1pt}$ that produced an escape, rounded to the nearest unit, MSec represents the average milleseconds execution of the algorithm rounded to one place, Max MSec represents the maximum millseconds over all trials, and Ascent represents the average number of times an escape along a codimension-$1$ face, i.e., an ascent,  was  found in the inner loop at (\ref{eq:qm1p4}), rounded to the nearest unit. 

Each row of each table represents an average (or max) over 25,000 trials,
not including the construction of the initial point on the boundary of the polytope, where the point $p$ is a pseudorandomly generated point 5 units from the barycenter of the polytope.

Table \ref{ta:qm6} represents trials in which the starting point $e_0$ is the unique $p$-visible point on the line between $p$ and the barycenter of the polytope.

 \begin{center}
\begin{table}[h!]
\refstepcounter{table}\label{ta:qm7}
\fontsize{8}{10}\selectfont
\centering
{\setlength{\tabcolsep}{.7em}
\begin{tabular}{|l|c|r|r|r|r|}
\hline
\bf Polyhedron & \bf Dimension  & \bf Steps & \bf MSec  & \bf Max MSec  & \bf Ascent \\ \hline
Cube       & 10  & 7    & 1.9  & 10  & 6                     \\ \hline
           & 20  & 14   & 24.5  & 58  & 13                     \\ \hline
           & 50 & 36   & 1046.7 & 2164  & 35                     \\ \hline
Simplex    & 10  & 6    & 1.1   & 3  & 4                     \\ \hline
           & 20  & 11    & 13.7  & 38  & 6                     \\ \hline
           & 50 & 25   & 616.8 & 1405   & 12                    \\ \hline
\end{tabular}}
\end{table}
\end{center}
\tcap{\ref{ta:qm7}.}{Vertex}

\vskip -5pt
Table \ref{ta:qm7} represents trials in which the starting point is the first $p$-visible vertex encountered in a pseudorandomly generated sequence of vertices of the polytope. It is clear that the scheme will be forced to search along codimension-$1$ faces in the case of a vertex starting point. This is a deliberately `bad' choice of starting point, since it incorporates more work to start with.
It seems particularly bad for the cube in $50$ dimensions, since there are $2^{50}$ vertices in play.

All computations for the results in this section were done on an ASUS laptop with 2.4Ghz processors, and written in the $C^{\#}$ language.

\vfill\break
\textbf{Appendix 1.}
Table \ref{ta:qm8}
\renewcommand{\arraystretch}{1.1}
\begin{center}
\begin{table}[h!]
\refstepcounter{table}\label{ta:qm8}
\fontsize{7}{9}\selectfont
\centering
{\setlength{\tabcolsep}{1em}
\begin{tabular}{|c|c|c|c|c|c|}
 \hline 
 \multicolumn{4}{|c|}{\textbf{Array}} & \multicolumn{1}{c|}{\textbf{Data}}  \\ \hline
        
	-1.0007 & -2.8066 & -5.8215 & 9.5544 & -28.4075   \\ \hline
	-14.6785 & -12.7791 & 4.1261 & 14.2914 & -94.2159  \\ \hline
	16.8958 & -5.0532 & 3.6602 & 6.2941 &  92.4505  \\ \hline
	-23.4113 & 19.7767 & 16.7373 & 1.1475 &  -51.9379  \\ \hline
	0.6277 & -5.3992 & 18.8794 & -13.4474 & 56.9866  \\ \hline
	-4.7236 & -14.6595 & 22.8559 & 19.2364 & 28.8286  \\ \hline
	-18.9662 & -8.4294 & 10.9331 & -20.0486 & -58.2480  \\ \hline
	-23.7980 & -1.3954 & 6.5481 & -18.9430 & -20.2276  \\ \hline
	20.4657 & -15.6368 & -23.8135 & 7.1174 & -19.8992  \\ \hline
	-19.4998 & 17.6706 & -16.7385 & -7.2366 & -31.6489  \\ \hline
	11.7233 & 23.8646 & 20.4484 & 15.1911 & 125.5107  \\ \hline
	-20.6015 & 3.3139 & -17.3742 & 11.0431 & -27.0742  \\ \hline
	24.9810 & -10.6350 & 9.8897 & 14.5601 & 22.7470  \\ \hline
	0.3252 & 9.0136 & -1.2913 & -20.9941 & 63.5170  \\ \hline
	2.9841 & -11.7714 & 10.6803 & 2.4431 & -32.8363  \\ \hline
	-8.7355 & -6.5857 & -13.2380 & 24.6572 & -16.5939  \\ \hline
	2.6058 & -22.9210 & 5.0925 & -9.8875 & -56.9162  \\ \hline
	-17.8308 & -3.8778 & 20.7420 & -6.1302 & -44.726  \\ \hline
	-2.1119 & -9.0212 & 20.5050 & -16.6813 & 43.5739  \\ \hline
	6.1007 & -0.5265 & -13.3736 & 2.5252 & 27.0157  \\ \hline
	-3.5134 & -6.0381 & -20.7444 & -5.1392 & -91.081  \\ \hline
	-1.5053 & -11.8158 & -6.9947 & 4.6541 & -45.507  \\ \hline
	10.9842 & -8.4481 & -17.0314 & 18.0321 & -10.7101  \\ \hline
	22.9020 & 9.2726 & -5.2228 & 1.5522 & 126.8742  \\ \hline
	-20.7970 & 13.2716 & -20.1611 & 6.1162 & 126.8742  \\ \hline 
	
\end{tabular}}
\end{table}
\end{center}
\tcap{\ref{ta:qm8}.}{Sample Data for Section 2.1}

\bibliographystyle{amsplain}

\begin{thebibliography}{99.}%

\bibitem {as-tlip}  
Schrijver, A.: Theory of Linear and Integer Programming. John Wiley \& Sons Ltd, (1986)

\bibitem {ms-lufam}  
Stromberg M. (2021) LU Factorization of Any Matrix. In: Baumann G. (eds) New Sinc Methods of Numerical Analysis: Festschrift in Honor of Frank Stenger's 80th Birthday. Trends in Mathematics. Birkhäuser, Cham. \verb| https://doi.org/10.1007/978-3-030-49716-3_14|
\end{thebibliography}

\end{document}